\documentclass[11pt]{amsart}

\pdfoutput=1
\usepackage[latin1]{inputenc}
\usepackage{subfigure, graphicx,  shuffle}
\usepackage{amssymb, amsmath}

\usepackage[colorlinks=true]{hyperref}

\numberwithin{equation}{section}
\newtheorem{theorem}{Theorem}

\newtheorem{lemma}[theorem]{Lemma}

\newtheorem{problem}[theorem]{Problem}

\newtheorem{defn}[theorem]{Definition}

\newtheorem{example}[theorem]{Example}

\newcommand{\Park}{\mathrm{Park}}
\newcommand{\CC}{\mathbb{C}}
\newcommand{\ZZ}{\mathbb{Z}}

\newcommand{\Frob}{\mathrm{Frob}}
\newcommand{\area}{\mathrm{area}}
\newcommand{\sign}{\mathrm{sign}}
\newcommand{\grFrob}{\mathrm{grFrob}}
\newcommand{\ex}{\mathrm{ex}}
\newcommand{\Res}{\mathrm{Res}}
\newcommand{\Ind}{\mathrm{Ind}}
\newcommand{\Lie}{\mathrm{Lie}}
\newcommand{\Sym}{\mathrm{Sym}}
\newcommand{\mult}{\mathrm{mult}}
\newcommand{\spn}{\mathrm{span}}

\title{Extending the parking space}

\author{Andrew Berget}
\address{Department of Mathematics\\University of Washington\\Seattle, WA, USA}
\email{aberget@math.washington.edu}

\author{Brendon Rhoades}
\address{Department of Mathematics\\University of California, San Diego\\La Jolla, CA, USA}
\email{bprhoades@math.ucsd.edu}


\begin{document}
\maketitle

\begin{abstract}
  The action of the symmetric group $S_n$ on the set $\Park_n$ of
  parking functions of size $n$ has received a great deal of attention
  in algebraic combinatorics.  We prove that the action of $S_n$ on
  $\Park_n$ extends to an action of $S_{n+1}$.  More precisely, we
  construct a graded $S_{n+1}$-module $V_n$ such that the restriction
  of $V_n$ to $S_n$ is isomorphic to $\Park_n$.  We describe the
  $S_n$-Frobenius characters of the module $V_n$ in all degrees and
  describe the $S_{n+1}$-Frobenius characters of $V_n$ in extreme
  degrees.  We give a bivariate generalization $V_n^{(\ell, m)}$ of
  our module $V_n$ whose representation theory is governed by a
  bivariate generalization of Dyck paths.  A Fuss generalization of
  our results is a special case of this bivariate generalization.
\end{abstract}

\section{Introduction}
\label{sec:in}
This paper is about extending the visible permutation action of $S_n$
on the space $\Park_n$ spanned by parking functions of size $n$ to a
hidden action of the larger symmetric group $S_{n+1}$.  The
$S_{n+1}$-module we construct will be a subspace of the coordinate
ring of the reflection representation of type A$_n$ and will inherit
the polynomial grading of this coordinate ring.  Using statistics on
Dyck paths, Theorem~\ref{maintheorem} will give an explicit
combinatorial formula for the graded $S_n$-Frobenius character of our
module. In Theorem~\ref{extremedegrees} we will describe the extended
$S_{n+1}$ action in extreme degrees.

As far as the authors know, this is the first example of an extension
of the $S_n$-module structure on $\Park_n$ to $S_{n+1}$ and the first
proof that the $S_n$-module structure on $\Park_n$ extends to
$S_{n+1}$. Our main theorems should be thought of as parallel to
several well known extensions, most notably the action of $S_{n+1}$ on
the multlinear subspace of the free Lie algebra on $n+1$ symbols,
which extends the regular representation of $S_n$ to $S_{n+1}$. See
Section~\ref{Concluding Remarks} for more on such questions.

\section{Background and Main Results}
\label{sec:main}

A length $n$ sequence $(a_1, \dots, a_n)$ of positive integers is
called a {\it parking function of size $n$} if its nondecreasing
rearrangement $(b_1 \leq \dots \leq b_n)$ satisfies $b_i \leq i$ for
all $i$\footnote{The terminology arises from the following situation.
  Consider a linear parking lot with $n$ parking spaces and $n$ cars
  that want to park in the lot.  For $1 \leq i \leq n$, car $i$ wants
  to park in the space $a_i$.  At stage $i$ of the parking process,
  car $i$ parks in the first available spot $\geq a_i$, if any such
  spots are available.  If no such spots are available, car $i$ leaves
  the lot.  The driver preference sequence $(a_1, \dots, a_n)$ is a
  parking function if and only if all cars are able to park in the
  lot.}.  Parking functions were introduced by Konheim and Weiss
\cite{KonheimWeiss} in the context of computer science, but have seen
much application in algebraic combinatorics with connections to
Catalan combinatorics, Shi hyperplane arrangements, diagonal
coinvariant rings, and rational Cherednik algebras.  The set of
parking functions of size $n$ is famously counted by $(n+1)^{n-1}$.
The $\CC$-vector space $\Park_n$ spanned by the set of parking
functions of size $n$ carries a natural permutation action of the
symmetric group $S_n$ on $n$ letters:
\begin{equation}
\label{eq:parkingaction}
w.(a_1, \dots, a_n) = (a_{w(1)}, \dots, a_{w(n)})
\end{equation}
for $w \in S_n$ and $(a_1, \dots, a_n) \in \Park_n$.
\footnote{Here we adopt the symmetric group multiplication convention that
says, for example, $(1,2)(2,3) = (1,2,3)$ so that this is a {\it left} action.}

A {\it partition} $\lambda$ of a positive integer $n$ is a weakly decreasing sequence 
$\lambda = (\lambda_1 \geq \dots \geq \lambda_k)$
of nonnegative integers which sum to $n$.  We write $\lambda \vdash n$ to 
mean that $\lambda$ is a partition of $n$ 
and define $|\lambda| := n$.  We call $k$ the {\it length} of the partition $\lambda$.
Observe that we allow zeros as parts of our partitions and that these zeros are included 
in the length.
The {\it Ferrers diagram} of $\lambda$ consists of $\lambda_i$ left justified boxes in the 
$i^{th}$ row from the top (`English notation').
If $\lambda$ is a partition, we define a new partition $\mult(\lambda)$ whose parts are obtained
by listing the (positive) part multiplicities in $\lambda$ in weakly decreasing order.
For example, we have that $\mult(4,4,3,3,3,1,0,0) = (3,2,2,1)$.

We will make use of two orders on partitions in this paper, one partial and one total.  The first partial order is
{\it Young's lattice} with relations given by $\lambda \subseteq \mu$ if $\lambda_i \leq \mu_i$ for all
$i \geq 1$ 
(where we append an infinite string of zeros to the ends of $\lambda$ and $\mu$ so that these
inequalities make sense). 
Equivalently, we have that $\lambda \subseteq \mu$ if and only if 
the Ferrers diagram of $\lambda$ fits inside the Ferrers diagram of $\mu$.  {\it Graded reverse
lexicographical (grevlex) order} 
is the total order on partitions of fixed length $n$ defined by 
$\lambda \prec \mu$ if either $|\lambda| < |\mu|$ or the 
final nonzero entry in the vector difference $\lambda - \mu$ is positive.
For example, if $n = 6$ we have $(4,2,2,2,1,0) \prec (3,3,3,1,1,0)$.
In particular, either 
of the relations $\lambda \subseteq \mu$ or $\lambda \preceq \mu$ imply that 
$|\lambda| \leq |\mu|$.

For a partition $\lambda = (\lambda_1, \dots, \lambda_k) \vdash n$, we
let $S_{\lambda}$ denote the {\it Young subgroup} $S_{\lambda_1}
\times \cdots \times S_{\lambda_k}$ of $S_n$.  We denote by
$M^{\lambda}$ the coset representation of $S_n$ given by $M^{\lambda}
:= \Ind_{S_{\lambda}}^{S_n} ({\bf 1}_{S_{\lambda}}) \cong_{S_n} \CC
S_n / S_{\lambda}$ and we denote by $S^{\lambda}$ the irreducible
representation of $S_n$ labeled by the partition $\lambda$.

Let $R_n$ denote the $\CC$-vector space of class functions $S_n
\rightarrow \CC$.  Identifying modules with their characters, the set
$\{ S^{\lambda} \,:\, \lambda \vdash n \}$ forms a basis of $R_n$.
The graded vector space $R := \bigoplus_{n \geq 0} R_n$ attains the
structure of a $\CC$-algebra via the induction product $S^{\lambda}
\circ S^{\mu} := \Ind_{S_n \times S_m}^{S_{n+m}}(S^{\lambda}
\otimes_{\CC} S^{\mu})$, where $\lambda \vdash n$ and $\mu \vdash m$.

We denote by $\Lambda$ the ring of symmetric functions (in an infinite
set of variables $X_1, X_2, \dots$, with coefficients in $\CC$).  The
$\CC$-algebra $\Lambda$ is graded and we denote by $\Lambda_n$ the
homogeneous piece of degree $n$.  Given a partition $\lambda$, we
denote the corresponding Schur function by $s_{\lambda}$ and the
corresponding complete homogeneous symmetric function by
$h_{\lambda}$.

The {\it Frobenius character} is the graded 
$\CC$-algebra isomorphism $\Frob: R \rightarrow \Lambda$
induced by setting $\Frob(S^{\lambda}) = s_{\lambda}$.  It is well known that
we have $\Frob(M^{\lambda}) = h_{\lambda}$.  Generalizing slightly, if $V = \bigoplus_{k \geq 0} V(k)$
is a graded $S_n$-module, define 
the {\em graded Frobenius character}
$\grFrob(V; q) \in \Lambda \otimes_{\CC} \CC[[q]]$
to be the formal power series in $q$ with coefficients in $\Lambda$ given by
$\grFrob(V; q) := \sum_{k \geq 0} \Frob(V(k)) q^k$.

A {\it Dyck path of size $n$} is a lattice path $D$ in $\ZZ^2$
consisting of vertical steps $(0, 1)$ and horizontal steps $(1, 0)$ which
starts at $(0, 0)$, ends at $(n, n)$, and stays weakly above the line $y = x$.  A maximal continguous
sequence of vertical steps in $D$ is called a {\it vertical run} of $D$. 

We will associate two partitions to a Dyck path $D$ of size $n$.  The
{\it vertical run partition} $\lambda(D) \vdash n$ is obtained by
listing the (positive) lengths of the vertical runs of $D$ in weakly
decreasing order.  For example, if $D$ is the Dyck path in
Figure~\ref{fig:dycksix}, then $\lambda(D) = (3,2,1)$.  The {\it area
  partition} $\mu(D)$ is the partition of length $n$ whose Ferrers
diagram is the set of boxes to the upper left of $D$ in the $n \times
n$ square with lower left coordinate at the origin.  For example, if
$D$ is the Dyck path of size $6$ in Figure~\ref{fig:dycksix}, then
$\mu(D) = (5,1,1,1,0,0)$.  The boxes in the Ferrers diagram of
$\mu(D)$ are shaded.  We define the {\it area} statistic\footnote{Many
  authors instead define the area of a Dyck path $D$ to be the number
  of complete lattice squares between $D$ and the line $y = x$, so
  that our statistic would be the `coarea'.} on Dyck paths by
$\area(D) = |\mu(D)|$.  For the Dyck path in our running exampe,
$\area(D) = 8$.  By construction, we have that $\mult(\mu(D)) =
\lambda(D)$ for any Dyck path $D$ of size $n$.

Dyck paths of size $n$ can be used to obtain a decomposition of
$\Park_n$ as a direct sum of coset modules $M^{\lambda}$.  In
particular, let $D$ be a Dyck path of size $n$.  A {\it labeling} of
$D$ assigns each vertical run of $D$ to a subset of $[n] := \{1, 2,
\dots, n \}$ of size equal to the length of that vertical run such
that every letter in $[n]$ appears exactly once as a label of a
vertical run.  Figure~\ref{fig:dycksix} shows an example of a labeled
Dyck path of size $6$, where the subsets labeling the vertical runs
are placed just to the right of the runs.

\begin{figure}
\centering
\includegraphics[scale=0.4]{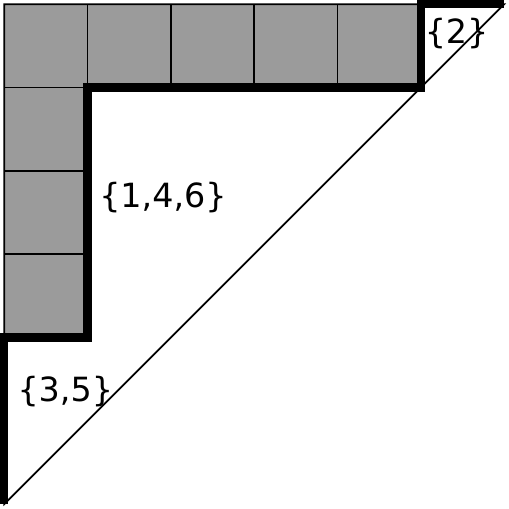}
\caption{A Dyck path of size $6$.}
\label{fig:dycksix}
\end{figure}

The set of labeled Dyck paths of size $n$ carries an action of $S_n$
given by label permutation.  There is an $S_n$-equivariant bijection
from the set of labeled Dyck paths $D$ of size $n$ to parking
functions $(a_1, \dots, a_n)$ of size $n$ given by letting $a_i$ be
one greater than the $x$-coordinate of the vertical run of $D$ labeled
by $i$.  For example, the labeled Dyck path in
Figure~\ref{fig:dycksix} corresponds to the parking function $(2, 6,
1, 2, 1, 2) \in \Park_6$.  Since any fixed labeled Dyck path of size
$D$ generates a cyclic $S_n$-module isomorphic to $M^{\lambda(D)}$, it
is immediate that the parking space $\Park_n$ decomposes into coset
representations as
\begin{equation}
\label{eq:cosetdecomposition}
\Park_n \cong_{S_n} \bigoplus_D M^{\lambda(D)},
\end{equation}
where the direct sum is over all Dyck paths $D$ of size $n$.
Equivalently, we have that the Frobenius character of $\Park_n$ is
given by $\Frob(\Park_n) = \sum_D h_{\lambda(D)}$.  For example, the 5
Dyck paths of size 3 shown in Figure~\ref{fig:dyckthree} lead to the
Frobenius character
\begin{equation}
\Frob(\Park_3) = h_{(3)} + 3 h_{(2,1)} + h_{(1,1,1)}.
\end{equation}

The vector space underlying the $S_{n+1}$-module which will extend
$\Park_n$ is a subspace of the polynomial ring $\CC[x_1, \dots,
x_{n+1}]$ in $n+1$ variables and first studied in the work of
Postnikov and Shapiro \cite{PostnikovShapiro}. Let $K_{n+1}$ denote
the complete graph on the vertex set $[n+1]$.  Given an edge $e = (i <
j)$ in $K_{n+1}$, we associate the polynomial weight $p(e) := x_i -
x_j \in \CC[x_1, \dots, x_{n+1}]$.  A subgraph $G \subseteq K_{n+1}$
(identified with its edge set) gives rise to the polynomial weight
$p(G) := \prod_{e \in G} p(e)$.  Following Postnikov and Shapiro, we
call a subgraph $G \subseteq K_{n+1}$ {\it slim} if the complement
edge set $K_{n+1} - G$ is a connected graph on the vertex set $[n+1]$.

\begin{defn}
Denote by $V_n$ the $\CC$-linear subspace of $\CC[x_1, \dots,
x_{n+1}]$ given by
\begin{equation}
V_n := \spn \{ p(G) \,:\, \text{$G$ is a slim subgraph of $K_{n+1}$} \}.
\end{equation}
Let $V_n(k)$ denote the homogeneous piece of $V_n$ of polynomial
degree $k$; the space $V_n(k)$ is spanned by those polynomials $p(G)$
corresponding to slim subgraphs $G$ of $K_{n+1}$ with $k$ edges.
\end{defn}

While the set of polynomials $\{ p(G) \,:\, G$ is a slim subgraph of
$K_{n+1} \}$ is linearly dependent in general, a basis for $V_n$ can
be constructed using standard matroid theoretic results
\cite[Proposition~9.4]{PostnikovShapiro}.  Fix a total order on the
edge set of $K_{n+1}$.  Given a spanning tree $T$ of $K_{n+1}$, the
{\it external activity} $\ex(T)$ of $T$ is the set of edges $e \in
K_{n+1}$ such that $e$ is the minimal edge of the unique cycle in $T
\cup \{ e \}$.  A basis of $V_n$ is given by 
\begin{equation*}
\{ p(K_{n+1}-(\ex(T) \cup T)) \,:\, \text{$T$ is a spanning tree of $K_{n+1}$} \}. 
\end{equation*}
It follows
immediately from Cayley's theorem that $\dim V_n = (n+1)^{n-1}$.
Aside from this dimension formula, we will make no further use
of this basis (or, indeed, any explicit basis) of $V_n$ for the rest of the paper.

Since the slimness of a subgraph is preserved under the action of
$S_{n+1}$ on the vertex set $[n+1]$ and $p(G)$ is homogeneous of
degree equal to the number of edges in $G$, it follows that $V_n =
\bigoplus_{k \geq 0} V_n(k)$ is a graded $S_{n+1}$-submodule of the
polynomial ring $\CC[x_1, \dots, x_{n+1}]$.  In fact, the space $V_n$
sits inside the copy of the coordinate ring of the reflection
representation of type A$_n$ sitting inside $\CC[x_1, \dots, x_{n+1}]$
generated by $x_i - x_{i+1}$ for $1 \leq i \leq n$.

The following result was conjectured by the first author.  We postpone
its proof, along with the proofs of the other results in this section,
to Section~\ref{proofs}.

\begin{theorem}
\label{maintheorem}
Embed $S_n$ into $S_{n+1}$ by letting $S_n$ act on the first $n$
letters. We have that
\begin{equation}
\label{main-iso}
\Res^{S_{n+1}}_{S_n} (V_n(k)) \cong_{S_n} \bigoplus_D M^{\lambda(D)},
\end{equation}
where the direct sum is over all Dyck paths of size $n$ and area $k$.
In particular, by Equation~\ref{eq:cosetdecomposition} we have that
\begin{equation}
\Res^{S_{n+1}}_{S_n}(V_n) \cong_{S_n} \Park_n.
\end{equation}
\end{theorem}

\begin{example}
In the case $n = 2$, 
Figure~\ref{fig:slim} shows that four slim
subgraphs of the complete graph $K_3$.  From left to right, the
corresponding polynomials are $1, x_1 - x_2, x_1 - x_3$, and $x_2 -
x_3$.  It follows that $V_2(0) = \spn \{ 1 \}$ and $V_2(1) = \spn \{
x_1 - x_2, x_1 - x_3, x_2 - x_3 \}$.  Observe that the graded
Frobenius character of $V_2$ is $\grFrob(V_2; q) = s_{(3)} q^0 +
s_{(2,1)} q^1$.  By the branching rule for symmetric groups (see
\cite{Sagan}), we have that $\grFrob(\Res^{S_3}_{S_2}(V_2); q) =
s_{(2)} q^0 + (s_{(2)} + s_{(1,1)}) q^1$.  Setting $q = 1$ yields
$\Frob(\Res^{S_3}_{S_2}(V_2)) = 2 s_{(2)} + s_{(1,1)}$, which agrees
with the Frobenius character of $\Park_2$.

\begin{figure}
\centering
\includegraphics[scale = 0.5]{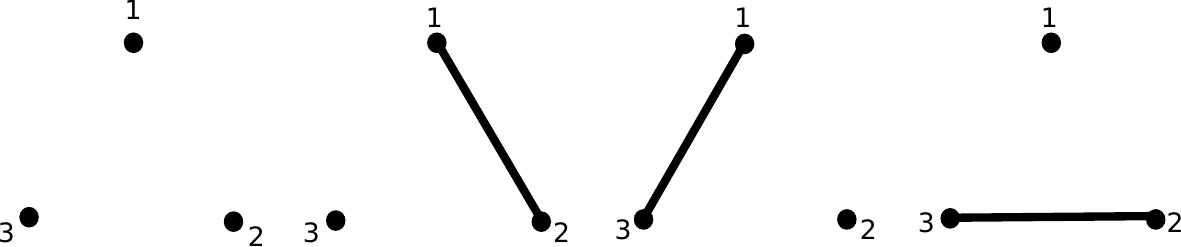}
\caption{The four slim subgraphs of $K_3$.}
\label{fig:slim}
\end{figure}
\end{example}

\begin{example}
Below we have the graded Frobenius character for $V_3$ and $V_4$.
\[
  \grFrob(V_3) = s_{(4)} + s_{(3,1)} q + (s_{(4)} + s_{(3,1)} +
  s_{(2,2)})q^2 + (s_{(3,1)} + s_{(2,1,1)} )q^3
\]
\begin{align*}
  \grFrob(V_4) &= s_{(5)} + s_{(4,1)}q + (s_{(5)} + s_{(4,1)} + s_{(3,2)})q^2\\ &\quad +
  (s_{(5)} + 2s_{(4,1)} + s_{(3,2)} + s_{(3,1,1)})q^3\\ &\ \quad+
  (s_{(5)} + 2s_{(4,1)} + 2s_{(3,2)} + s_{(3,1,1)} + s_{(2,2,1)})q^4\\ &\ \ \quad+
  (s_{(5)} + 2s_{(4,1)} + 2s_{(3,2)} + 2s_{(3,1,1)} + s_{(2,2,1)})q^5\\ &\ \ \ \quad+
  (s_{(4,1)} + s_{(3,2)} + s_{(3,1,1)} + s_{(2,2,1)} +s_{(2,1,1,1)})q^6.
\end{align*}
We leave it to the reader to check that the restrictions of these graded Frobenius 
characters yield the representations $\Park_3$ and $\Park_4$ of $S_3$ and $S_4$,
respectively.
\end{example}

Equivalently, we have that $\grFrob(\Res^{S_{n+1}}_{S_n}(V_n); q) =
\sum_D q^{\area(D)} h_{\lambda(D)}$, where the sum is over all Dyck
paths $D$ of size $n$.  For example, computing the area and run
partitions of the 5 Dyck paths of size 3 shown in
Figure~\ref{fig:dyckthree} shows that
\begin{equation}
\grFrob(\Res^{S_4}_{S_3}(V_3); q) = 
h_{(3)} q^0 + h_{(2,1)} q^1 + 2 h_{(2,1)} q^2 + h_{(1,1,1)} q^3.
\end{equation}

\begin{figure}
\centering
\includegraphics[scale=0.4]{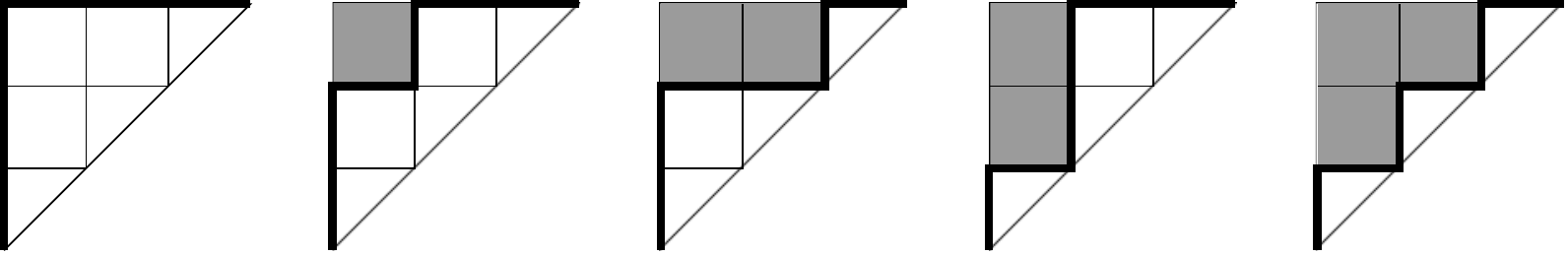}
\caption{The $5$ Dyck paths of size $3$.   From left to right, their contributions to the graded
Frobenius character
$\grFrob(\Res^{S_4}_{S_3}(V_3); q)$ are
$h_{(3)}q^0, h_{(2,1)} q^1, h_{(2,1)} q^2, h_{(2,1)}q^2,$ and $h_{(1,1,1)} q^3$.}
\label{fig:dyckthree}
\end{figure}

Postnikov and Shapiro showed that the dimension of the vector space
$V_n$ is equal to $(n+1)^{n-1}$, however the $S_n$-module structure of
$V_n$ has remained unstudied. Indeed, Theorem~\ref{maintheorem} is the
first description of the $S_n$-module structure of $V_n$. 

It is natural to ask for an explicit description of the
$S_{n+1}$-structure of $V_n$ or of its graded pieces $V_n(k)$.  This
problem is open in general, but we can describe the extended structure
of $V_n(k)$ in the extreme degrees $k = 0, 1, \dots, n-1$ as well as $k=
{n \choose 2}$.  Let $C_{n+1}$ be the cyclic subgroup of $S_{n+1}$
generated by the long cycle $c := (1, 2, \dots, n+1)$ and let $\zeta$
be the linear representation of $C_{n+1}$ which sends $c$ to
$e^{\frac{2 \pi i}{n+1}}$.  Mackey's Theorem can be used to prove that
the {\it Lie representation} $\Lie_n :=
\Ind_{C_{n+1}}^{S_{n+1}}(\zeta)$ of $S_{n+1}$ satisfies
$\Res^{S_{n+1}}_{S_n}(\Lie_n) \cong_{S_n} \CC[S_n]$.  Stanley proved
that the Lie representation arises as the action of $S_{n+1}$ on the
top poset cohomology of the lattice of set partitions of $[n+1]$,
tensored with the sign representation \cite{Stanley}.

\begin{theorem}
\label{extremedegrees}
The module $V_n(0)$ carries the trivial representation of $S_{n+1}$,
the module $V_n(1)$ carries the reflection representation of
$S_{n+1}$, and in general $V_n(k) = \operatorname{Sym}^k(V_n(1))$ for
$k < n$. The module $V_n( {n \choose 2} ) = V_n(\mathrm{top})$
carries the Lie representation of $S_{n+1}$ tensor the sign
representation.
\end{theorem}
The first part of this result is optimal in the sense that if $k \geq
n$ then $V_n(k)$ is a proper subspace of
$\operatorname{Sym}^k(V_n(1))$.

We will prove a bivariate generalization of Theorem~\ref{maintheorem}
which includes a `Fuss generalization' as a special case.  Given
$\ell, m, n > 0$, define a {\it $(\ell, m)$-Dyck path of size $n$} to
be a lattice path $D$ in $\ZZ^2$ consisting of vertical steps $(1,0)$
and horizontal steps $(0,1)$ which starts at $(-\ell + 1,0)$, ends at
$(mn, n)$, and stays weakly above the line $y = \frac{x}{m} $.  Taking
$\ell = m = 1$, we recover the classical notion of a Dyck path of size
$n$.  Taking $\ell = 1$ and $m$ general, the $(1, m)$-Dyck paths are
the natural Fuss extension of Dyck paths.  As before, we define the
{\it vertical run partition} $\lambda(D) \vdash n$ of an $(\ell,
m)$-Dyck path $D$ of size $n$ to be the partition obtained by listing
the lengths of the vertical runs of $D$ in weakly decreasing order.
We also define the {\it area partition} $\mu(D)$ to be the length $n$
partition whose Ferrers diagram fits between $D$ and a $(\ell - 1 +
mn) \times n)$ rectangle with lower left hand coordinate $(-\ell + 1,
0)$.  The {\it area} of $D$ is defined by $\area(D) := |\mu(D)|$.  We
have that $\mult(\mu(D)) = \lambda(D)$.

Figure~\ref{fig:extdyck} shows an example of a $(2,2)$-Dyck path of
size $3$.  The path $D$ starts at $(-1, 0)$, ends at $(6, 3)$, and
stays above the line $y = \frac{x}{2}$.  We have that $\lambda(D) =
(2,1) \vdash 3$, $\mu(D) = (5,1,1)$, and $\area(D) = 7$.

\begin{figure}
\centering
\includegraphics[scale=0.6]{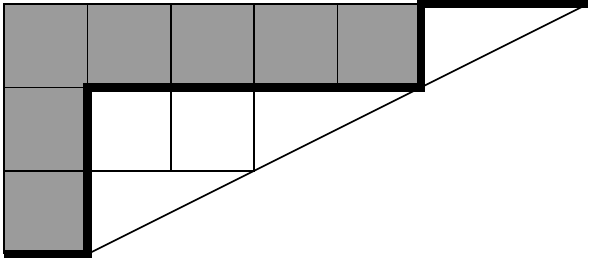}
\caption{A $(2,2)$-Dyck path of size $3$.}
\label{fig:extdyck}
\end{figure}

Let $K_{n+1}^{(\ell, m)}$ be the multigraph on the vertex set $[n+1]$
with $m$ edges between $i$ and $j$ for all $1 \leq i < j \leq n$ and
$\ell$ edges between $i$ and $n+1$ for all $1 \leq i \leq n$.  We call
a sub-multigraph $G$ of $K_{n+1}^{(\ell, m)}$ {\it slim} if the
multi-edge set difference $K_{n+1}^{(\ell, m)} - G$ is a connected
multigraph on $[n+1]$.  We extend the polynomial weight $p(G) \in
\CC[x_1, \dots, x_{n+1}]$ to multigraphs $G$ in the obvious way.

\begin{defn}
  Let $V_n^{(\ell, m)}$ be the $\CC$-linear subspace of $\CC[x_1,
  \dots, x_{n+1}]$ given by the span
\begin{equation}
V_n^{(\ell, m)} := \spn \{ p(G) \,:\, \text{$G$ is a slim sub-multigraph of $K^{(\ell, m)}_{n+1}$} \}.
\end{equation}
\end{defn}

As in the case $m = \ell = 1$, the space $V_n^{(\ell, m)}$ is stable
under the action of $S_n$, which respects the grading. When $\ell =
m$, $V_n^{(\ell,\ell)}$ also has an $S_{n+1}$-action, which also
respects the grading.  Postnikov and Shapiro showed that the dimension
of $V_n^{(\ell, m)}$ is $\ell (mn + \ell)^{n-1}$
\cite{PostnikovShapiro}.  Let $V_n^{(\ell, m)}(k)$ be the degree $k$
piece of $V_n^{(\ell, m)}$.

\begin{theorem}
\label{extension}
Under action of $S_n$ induced by permutation of vertex labels,
\begin{equation}
V_n^{(\ell, m)}(k) \cong_{S_n} \bigoplus_D
M^{\lambda(D)},
\end{equation}
where the direct sum is over all $(\ell, m)$-Dyck paths of size $n$
and area $k$. The containment $V_n^{(\ell, m)}(k) \subseteq
\Sym^k(V_n^{(\ell,m)}(1) )$ is an equality for $k < n$.

In the case $m = \ell$, the module $V_n^{(\ell,\ell)}(1)$ has
$S_{n+1}$-structure given by the reflection representation of
$S_{n+1}$, so that the equality $V_n^{(\ell, m)}(k) =
\Sym^k(V_n^{(\ell,m)}(1) )$, $k<n$, describes the $S_{n+1}$-module
structure completely.  The top degree space $V_n^{(\ell,\ell)}(\textup{top})$ is $\Lie_n \otimes (\sign)^{\otimes \ell}$.
\end{theorem}

\begin{example}
Take $n=3$, $\ell = m = 2$ in Theorem~\ref{extension}, so that $V_3^{(2,2)}$ carries an $S_4$ action. 
We have that
\begin{align*}
\grFrob(\Res^{S_4}_{S_3}(V_3^{(2,2)}) &= 
h_{(3)} q^0 + h_{(2,1)} q^1 + 2 h_{(2,1)} q^2 + (h_{(3)} + h_{(2,1)} + h_{(1,1,1)})q^3 \\
 &\quad\ + (3h_{(2,1)} + h_{(1,1,1)})q^4 + (3 h_{(2,1)} + 2 h_{(1,1,1)}) q^5 
 \\&\quad\ \  +(2 h_{(2,1)} + 3 h_{(1,1,1)}) q^6 
 + (2 h_{(2,1)} + 3 h_{(1,1,1)}) q^7\\ &\quad \ \ \  + 3 h_{(1,1,1)} q^8 + h_{(1,1,1)} q^9.
\end{align*}
\end{example}

\section{Proofs}
\label{proofs}

While Theorem~\ref{extension} implies Theorem~\ref{maintheorem}, the proof of 
Theorem~\ref{extension} is a straightforward extension of the proof of Theorem~\ref{maintheorem}
and it will be instructive to prove Theorem~\ref{maintheorem} first.

The first step in the proof of Theorem~\ref{maintheorem} is to relate the modules on both sides
of the claimed isomorphism by associating a subgraph $G(D)$ of $K_{n+1}$ and a polynomial
$p(D) \in \CC[x_1, \dots, x_{n+1}]$ to any Dyck path $D$ of size $n$.  We start by labeling the
$1 \times 1$ box $b$ 
which is completely above the line $y = x$ with the edge $e(b) = (n-j ,n-i)$ in
$K_{n+1}$, where $(i, j)$ is the upper left coordinate of $b$.  See
Figure~\ref{fig:dyckgraph} for an example of this labeling in the case $n = 5$.
We let $G(D)$ be the subgraph of $K_{n+1}$ consisting of those edges $e(b)$ for
which the box $b$ is to the upper left of the path $D$.  In Figure~\ref{fig:dyckgraph}, the shaded
boxes above the path $D$ each contribute an edge to the subgraph $G(D)$ and we have that 
$G(D) = \{ 1-6, 1-5, 1-4, 1-3, 2-6, 2-5, 3-6 \}$.

\begin{figure}
\centering
\includegraphics[scale = 0.6]{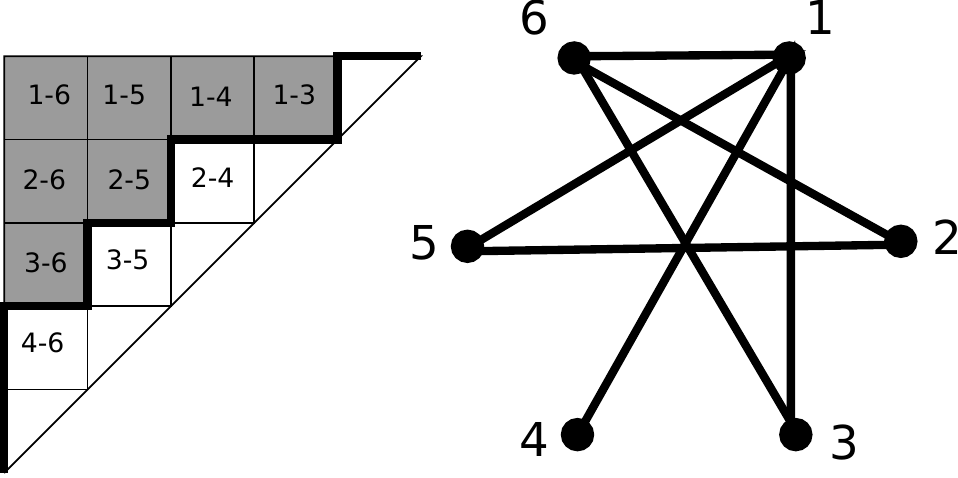}
\caption{A Dyck path $D$ of size $5$ and the associated subgraph $G(D)$ of $K_6$.}
\label{fig:dyckgraph}
\end{figure}

\begin{lemma}
\label{slim}
The subgraph $G(D)$ is slim for any Dyck path $D$.
\end{lemma}
\begin{proof} The subgraph $G(D)$ contains none of the edges in the path
$1 - 2 - \cdots - (n+1)$.
\end{proof}

By Lemma~\ref{slim}, the polynomial $p(D) := p(G(D))$ is contained in $V_n$.  
For example, if $n = 5$ and $D$ is the Dyck path shown in Figure~\ref{fig:dyckgraph}, we have that 
\begin{equation}
\label{examplepolynomial}
p(D) = (x_1 - x_6) (x_1 - x_5) (x_1 - x_4) (x_1 - x_3) (x_2 - x_6) (x_2 - x_5) (x_3 - x_6) \in V_5.
\end{equation}
By construction, for any Dyck path $D$ the polynomial $p(D)$ is homogeneous with degree equal
to $\area(D)$.

In order to prove the direct sum decomposition in Theorem~\ref{maintheorem}, 
we will show that the polynomials $p(D)$ project nicely onto 
a certain subspace of $\CC[x_1, \dots, x_{n+1}]$.
Since Theorem~\ref{maintheorem} only concerns the restriction of $V_n$ to $S_n$, it is natural
to consider a subspace of $\CC[x_1, \dots, x_{n+1}]$ which is closed under the action of 
$S_n$ but not of $S_{n+1}$.  

Let $st_n := (n-1, n-2, \dots, 1, 0)$ be the {\it staircase partition} of length $n$.
We call a partition $\lambda = (\lambda_1, \dots, \lambda_n)$  {\em sub-staircase} if 
$\lambda \subseteq st_n$ (observe that this definition has tacit dependence on $n$).
For any Dyck path $D$ of size $n$,
the partition $\mu(D)$ is sub-staircase.

For a partition $\lambda = (\lambda_1, \dots, \lambda_n)$, we use the shorthand 
$x^{\lambda} := x_1^{\lambda_1} \cdots x_n^{\lambda_n} \in \CC[x_1, \dots, x_n]$.
We call a monomial $x_1^{d_1} \cdots x_{n+1}^{d_{n+1}}$ in
the variables $x_1, \dots, x_{n+1}$ {\it sub-staircase} if there exists a permutation
$w \in S_n$ and a sub-staircase
partition $\lambda = (\lambda_1 \geq \dots \geq \lambda_n)$ such that
\begin{equation}
\label{staircasedefinition}
x_1^{d_1} \cdots x_{n+1}^{d_{n+1}} = w.x^{\lambda}.
\end{equation}
In particular, the variable $x_{n+1}$ does not appear in any sub-staircase monomial.  
If the monomial $x_1^{d_1} \cdots x_{n+1}^{d_{n+1}}$ is sub-staircase, the partition
$\lambda$ is uniquely determined from the monomial;  
call this the 
{\it exponent partition} of the monomial.  More generally, if 
$x_1^{a_1} \cdots x_n^{a_n}$ is any monomial in $x_1, \dots, x_n$, we call the 
rearrangement $(\lambda_1 \geq \dots \geq \lambda_n)$ of $(a_1, \dots, a_n)$ the
{\it exponent partition} of this monomial.
Let $W_n \subset \CC[x_1, \dots, x_{n+1}]$ 
be the $\CC$-linear span of all sub-staircase monomials.  The subspace $W_n$ is closed under 
the action of $S_n$, but not under the action of $S_{n+1}$.

In the case $n = 3$, the $S_3$-orbits of the $16$ sub-staircase monomials in $\CC[x_1, \dots, x_4]$ are 
shown in the following table, where the left column shows a representative from each orbit.

\begin{center}
\begin{tabular}{| l || l |}
\hline
$1$ & \\
$x_1$ & $x_2, x_3$ \\
$x_1^2$ & $x_2^2, x_3^2$ \\
$x_1 x_2$ & $x_1 x_3, x_2 x_3$ \\
$x_1^2 x_2$ & $x_1^2 x_3, x_2^2 x_1, x_2^2 x_3, x_3^2 x_1, x_3^2 x_2$\\
\hline 
\end{tabular}
\end{center}

The $S_3$-orbits are parametrized by sub-staircase partitions 
$\lambda = (\lambda_1, \lambda_2, \lambda_3)$ and
each orbit contains a unique representative of the form
$x^{\lambda}$.  The staircase monomials
form a linear basis of $W_3$ and the cyclic $S_3$-submodule of $W_3$ generated 
by $x^{\lambda}$ is isomorphic to $M^{\mult(\lambda)}$.  
The natural bijection between exponent vectors and parking functions affords an isomorphism
$W_3 \cong_{S_3} \Park_3$.
These observations generalize in a 
straightforward way to the following lemma, whose proof is left to the reader.

\begin{lemma}
\label{substaircasemodule}
The set of sub-staircase monomials forms a linear basis for $W_n$ and is closed under the action
of $S_n$.  The $S_n$-orbits are parametrized by sub-staircase partitions $\lambda$, and
the orbit labeled by $\lambda$ has a unique monomial of the form $x^{\lambda}$.  The cyclic 
$S_n$-submodule of $W_n$ generated by $x^{\lambda}$ is isomorphic to 
$M^{\mult(\lambda)}$ and we have that 
$W_n \cong_{S_n} \Park_n$.
\end{lemma}

With Lemma~\ref{substaircasemodule} in mind, we will construct a graded 
$S_n$-module isomorphism $V_n \xrightarrow{\sim} W_n$.
We define a graded
$S_n$-module homomorphism $\phi: V_n \rightarrow W_n$ by the following 
composition:
\begin{equation}
\label{maps}
\phi: V_n \hookrightarrow \CC[x_1, \dots, x_{n+1}] \twoheadrightarrow 
\CC[x_1, \dots, x_n] \twoheadrightarrow W_n,
\end{equation}
where the first map is inclusion, the second is the specialization $x_{n+1} = 0$, and the third linear 
map fixes the space $W_n$ pointwise and sends monomials which are not sub-staircase to zero.

We want to show that $\phi$ is an isomorphism.
Postnikov and Shapiro showed that $\dim(W_n) = \dim(V_n) = (n+1)^{n-1}$ 
\cite{PostnikovShapiro},
so it is enough 
to show that $\phi$ is surjective.  We will do this by  analyzing 
the polynomials $\phi(p(D))$, where $D$ is a Dyck path of size $n$.

The next lemma states that the transition matrix between the set
$\{ \phi(p(D)) \,:\, \text{$D$ a Dyck path of size $n$} \}$ expands in the monomial basis of $W_n$ given by
$\{ x^{\lambda} \,:\, \text{$\lambda$ sub-staircase} \}$ in a unitriangular way with respect
to  grevlex order (where we associate $\phi(p(D))$ with the  partition
$\mu(D)$).
The authors find it surprising that the corresponding unitriangularity statement is 
false when one considers linear extensions of the dominance order on partitions; 
grevlex order is primarily known for its utility in the efficient computation of Gr\"obner bases
and is far less ubiquitous in combinatorial representation theory than dominance order.

\begin{lemma}
\label{triangularity}
Let $D$ be a Dyck path of size $n$.  There exist integers $c_{\lambda, w} \in \ZZ$ such that
\begin{equation}
\phi(p(D)) = x^{\mu(D)} + \sum_{\substack{\lambda \prec \mu(D) \\ |\lambda| = |\mu(D)| \\
w \in S_n }}
c_{\lambda, w} w.x^{\lambda}.
\end{equation}
\end{lemma}
\begin{proof}
By definition, we have that
\begin{equation}
p(D) = \prod_{e = (i < j) \in G(D)} (x_i - x_j).
\end{equation}
Up to sign, a typical monomial in this expansion can be obtained by choosing an orientation
$\mathcal{O}$ of the graph $G(D)$, associating an oriented edge $k \rightarrow \ell$ to the variable
$x_{\ell}$, and multiplying the corresponding variables together.  The map $\phi$ kills any monomial
which contains the variable $x_{n+1}$, so up to sign a typical monomial in $\phi(p(D))$ is obtained
by choosing an orientation $\mathcal{O}$ of $G(D)$ such that every edge which contains the vertex
$n+1$ is oriented away from $n+1$.  If we denote by $\mathcal{O}_0$ the orientation 
of $G(D)$ which directs every edge towards its smaller endpoint, we see that the monomial
$x^{\mu(D)}$ arises in this way (with coefficient $1$).  We argue that there does not exist an orientation
$\mathcal{O}$ of $G(D)$ other than $\mathcal{O}_0$ such that the monomial $m$ corresponding to
$\mathcal{O}$ has exponent partition $\lambda$ satisfying $\lambda \succeq \mu(D)$.

Suppose for the sake of contradiction 
that there was an orientation $\mathcal{O}$ of $G(D)$  directing
incident edges away from the vertex $n+1$
other than the orientation $\mathcal{O}_0$ 
which gave rise to a monomial
$m$ in the variables $x_i$  with exponent partition $\lambda$ where $\mu(D) \preceq \lambda$.
If $\mu(D)$ were the empty partition, then $G(D)$ would be the empty graph and $\mathcal{O} = \mathcal{O}_0$,
so we conclude that $\mu(D)$ must have at least one nonempty column.
For $1 \leq i \leq n$, let $\mu(D)'_i$ denote the number of parts of the partition $\mu(D)$ which are $\geq i$.
Then the number of edges of $G(D)$ that contain $n+1$ equals $\mu(D)'_1$.
 By the way we labeled our boxes, this means that the product
$x_1 \cdots x_{\mu(D)'_1}$ divides $m$ and that the exponent partition $\lambda$ of $m$ is componentwise
 $\geq (1, 1, \dots, 1, 0, \dots, 0)$, where there are $\mu(D)'_1$ copies of  $1$ and $n - \mu(D)'_1$ 
copies of $0$.  
If $\mu(D)$ consists of a single nonempty column, we have $m = x_1 \dots x_{\mu(D)'_1} = x^{\mu(D)}$,
$\mathcal{O} = \mathcal{O}_0$,
and our contradiction.
Otherwise,
since $D$ is a Dyck path, we have that $n - \mu(D)'_1 > 0$.  In particular,
the condition $\mu(D) \preceq \lambda$ implies that none of the nonempty list of variables $x_{\mu(D)'_1 + 1},
x_{\mu(D)'+2}, \dots, x_n$ appear in $m$. In particular,
this means that the orientation $\mathcal{O}$ directs every edge away from the vertex $n$
(there are edges which are incident to $n$ because of the way we labeled our boxes and the fact that
$\mu(D)$ has at least two nonempty columns).

In general, suppose that the orientations $\mathcal{O}$ and $\mathcal{O}_0$ do not agree
and choose an edge $\{i, j\}$ with $i < j$ of $G(D)$ such that $\mathcal{O}$ and $\mathcal{O}_0$ disagree
on $\{i, j\}$ but agree on all edges $\{k, \ell\}$ with $k < \ell$ such that $\ell > j$ or $\ell = j$ but $k < i$.
By the reasoning of the last paragraph, we must have $j < n$.  Moreover, 
by considering the edges  of $G(D)$ with larger vertex $n+1, n, n-1, \dots, j+1$, we see that
$m$ is divisible by the monomial
\begin{equation*}
m' := (x_1 \cdots x_{\mu(D)'_1}) (x_1 \cdots x_{\mu(D)'_2}) \cdots (x_1 \cdots x_{\mu(D)'_{n-j+1}}).
\end{equation*}
This means that the exponent partition $\lambda$ is componentwise
$\geq$ the partition $(n-j+1, \dots, n-j+1, n-j, \cdots, n-j, \dots, 1, \dots, 1, 0, \dots, 0)$, where there
are $\mu(D)'_{n-j+1}$ copies of $n-j+1$, $\mu(D)'_{n-j} - \mu(D)'_{n-j+1}$ copies
of $n-j$, $\dots$, $\mu(D)_1' - \mu(D)_2'$ copies of $1$, and $n - \mu(D)'_1$ copies of $0$.  
In particular, the subsequence of this partition after the initial string of $(n-j+1)$'s agrees with 
the partition $\mu(D)$.
The 
fact that $D$ is a Dyck path implies that
the list of variables $x_{\mu(D)'_{n-j+1}+1}, x_{\mu(D)'_{n-j+1}+2}, \dots, x_j$ is nonempty and
the fact that $\lambda \preceq \mu(D)$ implies that none of these variables divides the 
quotient $\frac{m}{m'}$.  But the fact that $\mathcal{O}$ directs $i \rightarrow j$ implies that 
$x_j$ divides this quotient, which is a contradiction.
We conclude that $m = x^{\mu(D)}$ and
 $\mathcal{O} = \mathcal{O}_0$.
\end{proof}

As an example of Lemma~\ref{triangularity}, consider the case $n = 5$
and let the Dyck path $D$ be shown in Figure~\ref{fig:dyckgraph} with
$\mu(D) = (4,2,1,0,0)$.  To calculate $\phi(p(D))$, we set $x_6 = 0$ in
the product formula for $p(D)$ given in
Equation~\ref{examplepolynomial} and expand.  The resulting polynomial
is
\begin{align*}
\label{examplephi}
\phi(p(D)) &= x_1 (x_1 - x_5) (x_1 - x_4) (x_1 - x_3) x_2 (x_2 - x_5) x_3
= x_1^4 x_2^2 x_3 + \cdots .
\end{align*}
where the elipsis denotes terms involving sub-staircase monomials with
exponent partition $\prec (4,2,1,0,0)$.  We are ready to complete the
proof of Theorem~\ref{maintheorem}.

\begin{proof}[Proof of Theorem~\ref{maintheorem}]
By Lemma~\ref{substaircasemodule}, the set of sub-staircase
monomials forms a linear basis of $W_n$, so
Lemma~\ref{triangularity} implies that the $S_n$-module homomorphism
$\phi: V_n \rightarrow W_n$ is surjective.  Since $\dim(V_n) =
\dim(W_n)$, this implies that $\phi$ is also injective and gives an
isomorphism $\Res^{S_{n+1}}_{S_n}(V_n) \cong_{S_n} \Park_n$.  To
prove the graded isomorphism in Theorem~\ref{maintheorem}, it is
enough to observe that $\mult(\mu(D)) = \lambda(D)$ for any Dyck
path $D$ and apply Lemmas~\ref{substaircasemodule} and
~\ref{triangularity} together with the fact that $\phi$ is graded.
\end{proof}

It may be tempting to guess that $p(D)$ generates a cyclic
$S_n$-submodule of $V_n$ isomorphic to $M^{\lambda(D)}$, but this is
false in general.  The reason for this is that while the `leading
term' in the expansion of $\phi(p(D))$ in Lemma~\ref{triangularity}
generates the submodule $M^{\lambda(D)}$ under the action of $S_n$,
the other terms in this expansion can cause $\phi(p(D))$ to generate a
different cyclic submodule.

We are ready to prove the claimed $S_{n+1}$-structure of the extreme
degrees of the graded module $V_n(k)$.

\begin{proof}[Proof of Theorem~\ref{extremedegrees}]
  It is clear from the definitions that $V_n(0)$ carries the trivial
  representation of $S_{n+1}$.  The space $V_n(1)$ has basis given by
  the polynomials $x_1 - x_2, x_2 - x_3, \dots, x_n - x_{n+1}$ and
  hence carries the reflection representation of $S_{n+1}$ (i.e., the
  irreducible $S_{n+1}$-module corresponding to the partition
  $(n,1)$). Since $V_n \subseteq \Sym( V_n(1) )$ we are claiming that
  in degree $k < n$ this is an equality. The Hilbert series of
  $V_n$ is the Tutte polynomial evaluation $q^{\binom{n+1}{2} -
    n}T_{K_{n+1}}(1,1/q)$ and so we must prove that the first $n-1$
  terms of this sum are the binomial coefficents $\binom{n+k-1}{k}$.
  There is nothing special about $K_{n+1}$ in this claim and we will
  prove a more general statement in Lemma~\ref{lem:tutteCoeff}.

  To prove that $V_n(\mathrm{top})$ is isomorphic to $\Lie_{n+1}
  \otimes \sign$ we reason as follows. The space $V_n(\mathrm{top})$
  is spanned by those $p(G)$ where the complementary subgraph $K_{n+1}
  \setminus G$ is connected and has $n$ edges.

  Let $\mathcal{A}_n$ denote the braid arrangementin $\CC^{n+1}$,
  which is the union of those hyperplanes with at least two
  coordinates equal. Let $H^*( \CC^{n+1} \setminus \mathcal{A}_n;\CC)$
  denote the (complexified) de Rham cohomology ring of its
  complement. Consider, now, the linear map $c:V_n(\mathrm{top}) \to
  H^{n}( \CC^{n+1} \setminus \mathcal{A}_n )$ that sends
  \[
  p(G) \mapsto  p(G) \cdot{d(x_1-x_2) \wedge d(x_2-x_3) \wedge \dots \wedge d(x_n
    -x_{n+1})}/\prod_{1 \leq i < j \leq n} (x_i-x_j).
  \]
  This is an isomorphism of vector spaces, since it is division by the
  Vandermonde product, followed by multiplication by the $n$-form. To
  see that $c$ is equivariant notice that $\bigwedge^n V_n(1)$ carries
  the sign representation of $S_{n+1}$, because it is $1$ dimensional
  and non-trivial. Likewise does the one dimensional representation
  spanned by the Vadnermonde product. It follows that the signs
  introduced by multiplication by the $n$-form and division by the
  Vandermonde cancel, and $c$ is equivariant.

  Finally, the top degree cohomology of the complement $\CC^{n+1}
  \setminus \mathcal{A}_n$ is known to be isomorphic to the top
  degreee \textit{Whitney homology} of its lattice of flats
  \cite[Theorem~7.2.10]{bj}, and this correspondence is at once seen
  to be $S_{n+1}$-equivariant. The lattice of flats of $\mathcal{A}_n$
  is the partition lattice $\Pi_{n+1}$ and by a result of Stanley
  \cite{Stanley} (beautifuly recounted by Wachs in
  \cite[Section~4.4]{Wachs}), the top degree Whitney homology of the
  partition lattice $\Pi_{n+1}$ is $\Lie_{n+1} \otimes \sign$.
\end{proof}

\begin{lemma}\label{lem:tutteCoeff}
  Let $G$ be a connected graph on $v$ vertices with $e$ edges. Denote
  the Tutte polynomial of $G$ by $T_G(x,y)$. Then,
  the polynomial $q^{e-v+1}T_G(1,1/q)$ takes the form,
  \[
  1+  (v-1) q +
  \binom{v}{2}q^2 + \binom{v+1}{3}q^3 +\dots+
  \binom{(v-1) + (v-2) - 1}{v-2}q^{v-2} + O(q^{v-1}).
  \]
\end{lemma}
\begin{proof} We write $T_G(x,y)$ in terms of the two variable
  \textit{coboundary polynomial},
  $\overline{\chi}_G(\lambda,\nu)$. This is the sum
  \[
  \overline{\chi}_G(\lambda,\nu) = \frac{1}{\lambda}\sum_{i=0}^e
  c_i(G;\lambda) \nu^i
  \]
  where $c_i(G;\lambda)$ is the number of ways to color the vertices
  of $G$ with $\lambda$ colors and exactly $i$ monochromatic edges. It
  is a fact that this is a polynomial in $\lambda$ and $\nu$. Now by
  \cite[Proposition~6.3.26]{matroids},
  \[
  q^{e-v+1}T_G(1,1/q) = \frac{q^e}{(1-q)^{v-1}} \overline{\chi}_G(0,1/q).
  \]
  Thus, to prove the first part of the lemma we will show that
  $c_i(G;\lambda) = 0$ for $e-v+1 < i < e$, and that $c_e(G;\lambda) =
  \lambda$. Suppose that we have colored the vertices of $G$ and we
  have more than $e-v+1$ monochromatic edges. Then the collection of
  monochromatic edges forms a connected subgraph of $G$. It follows
  that all vertices of $G$ are colored the same and hence all edges of
  $G$ are monochromatic. This means that $c_i(G;\lambda) = 0$ unless
  $i = e$. That $c_e(G;\lambda) = \lambda$ is clear.
\end{proof}

The proof of Theorem~\ref{extension} is a straightforward extension
of the proof of Theorem~\ref{maintheorem}.  We will be somewhat brief.

\begin{proof}[Proof of Theorem~\ref{extension}]
Given any $(\ell, m)$-Dyck path $D$ of size $n$
we associate a sub-multigraph $G(D)$ of $K_{n+1}^{(\ell, m)}$ by letting every box which contributes
to $\area(D)$ correspond to a single edge in the multigraph $G(D)$;  the labeling which 
accomplishes this is shown in Figure~\ref{fig:extdyckgraph} in the case
$(\ell, m) = (3, 2)$ and $n = 4$.  For general $\ell, m,$ and $n$, we label the boxes in the 
$i^{th}$ row from the top from left to right with
$(\ell + m - 2)$ copies of the edge $i - (n+1)$, 
$m$ copies of the edge $i - n$, 
$m$ copies of the edge $i - (n-1), \dots,$
$m$ copies of the edge $i - (i+2)$, and
$(m-1)$ copies of the edge $i - (i+1)$.

For any $(\ell, m)$-Dyck path $D$ of size $n$, the multigraph complement of $G(D)$ within
$K^{(\ell, m)}_n$ contains each of the edges in the path $1 - 2 - \dots - n - (n+1)$ with 
multiplicity at least one.
Therefore, the sub-multigraph $G(D)$ is slim and the polynomial
$p(D) := p(G(D))$ is contained in $V^{(\ell, m)}_n$.

\begin{figure}
\centering
\includegraphics[scale = 0.6]{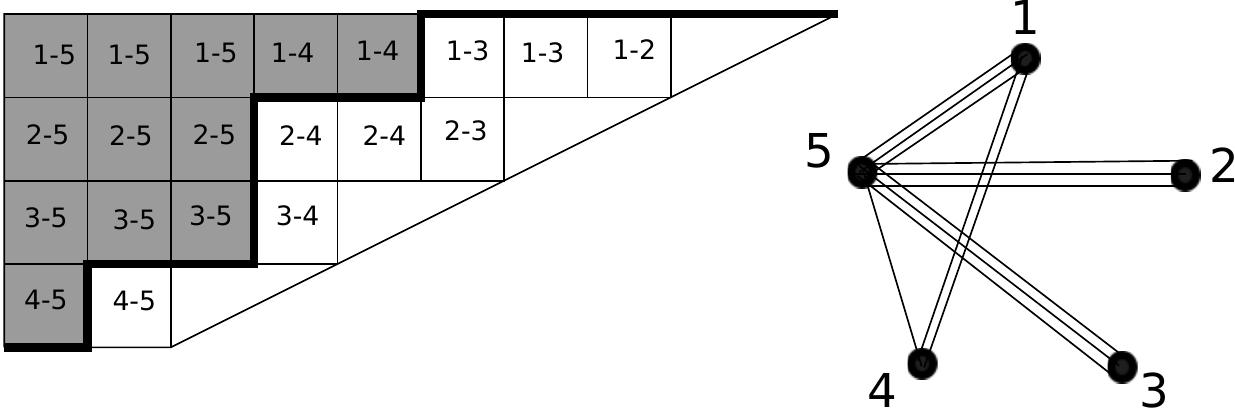}
\caption{A $(3,2)$-Dyck path $D$ of size $4$ and the associated sub-multigraph $G(D)$ of
$K^{(3,2)}_4$.}
\label{fig:extdyckgraph}
\end{figure}

We say that a partition $\lambda$ with $n$ parts is  
{\it sub-$(\ell, m)$-staircase} 
 if in Young's lattice we have the relation
 $\lambda \subseteq (\ell - 1 + m(n-1), \ell - 1 + m(n-2), \dots, \ell -1)$. 
A monomial $x_1^{d_1} \cdots x_{n+1}^{d_{n+1}}$ is 
{\it sub-$(\ell, m)$-staircase} if there exists $w \in S_n$ and a sub-$(\ell, m)$-staircase
partition $\lambda$ such that 
\begin{equation}
x_1^{d_1} \cdots x_{n+1}^{d_{n+1}} = x_{w(1)}^{\lambda_1} \cdots x_{w(n)}^{\lambda_n}.
\end{equation}
In particular, the variable $x_{n+1}$ does not appear in any sub-$(\ell, m)$-staircase
monomial.

Let $W^{(\ell, m)}_n$ be the subspace of $\CC[x_1, \dots, x_{n+1}]$ spanned by the set of all
sub-$(\ell, m)$-staircase monomials.   
The space $W^{(\ell, m)}_n$ carries a graded action of $S_n$.  The argument used to prove
Lemma~\ref{substaircasemodule} extends to show that 
the degree $k$ homogeneous piece of $W^{(\ell, m)}_n$ is isomorphic as an
$S_n$-module to the direct sum on the right hand side of the isomorphism asserted in
Theorem~\ref{extension}.

The isomorphism in Theorem~\ref{extension} 
is proven by showing that the graded $S_n$-module homomorphism
$\phi^{(\ell, m)}: V^{(\ell, m)}_n \rightarrow W^{(\ell, m)}_n$ given by the composite
\begin{equation}
\phi^{(\ell, m)}: V^{(\ell, m)}_n \hookrightarrow \CC[x_1, \dots, x_{n+1}] \twoheadrightarrow
\CC[x_1, \dots, x_n] \twoheadrightarrow W^{(\ell, m)}_n
\end{equation}
is an isomorphism, where the first map is inclusion, the second is the evaluation $x_{n+1} = 0$, and
the third fixes $W^{(\ell, m)}_n$ pointwise and sends every monomial which is not
sub-$(\ell, m)$-staircase to zero.  

Postnikov and Shapiro proved that the vector space $V^{(\ell, m)}_n$ has dimension 
$\ell (\ell + mn)^{n-1}$ \cite{PostnikovShapiro}.   
Pitman and Stanley \cite{PitmanStanley} and Yan \cite{Yan} showed that the number of
(exponent vector of) sub-$(\ell, m)$-staircase monomials equals 
$\ell (\ell + mn)^{n-1}$.  Since these monomials form a basis
for $W^{(\ell, m)}_n$, 
 in order
to prove that $\phi^{(\ell, m)}$ is an isomorphism of $S_n$-modules, it is enough to show that
$\phi^{(\ell, m)}$ is surjective.

The fact that $\phi^{(\ell, m)}$ is surjective follows from the following triangularity result which
generalizes Lemma~\ref{triangularity}.  Recall that $\mu(D)$ is the partition whose Ferrers
diagram lies to the northwest of an 
$(\ell, m)$-Dyck path $D$ (for example, if $D$ is the 
$(3, 2)$-Dyck path appearing on the left in Figure~\ref{fig:extdyckgraph}, then
$\mu(D) = (5,3,3,1)$.  The proof of Lemma~\ref{extendedtriangularity} is 
almost identical to the proof of Lemma~\ref{triangularity} and is left to the reader.

\begin{lemma}
\label{extendedtriangularity}
Let $D$ be an $(\ell, m)$-Dyck path of size $n$.
The monomial expansion of $\phi^{(\ell, m)}(p(D))$ has the form
\begin{equation}
\phi^{(\ell, m)}(p(D)) = 
x^{\mu(D)} + \cdots,
\end{equation}
where the elipsis denotes terms involving monomials whose exponent
partitions are $\prec \mu(D)$. 
\end{lemma}
For example, if $D$ is the $(3, 2)$-Dyck
path in Figure~\ref{fig:extdyckgraph}, then
\begin{equation*}
p(D) = (x_1 - x_4)^2 (x_1 - x_5)^3 (x_2 - x_5)^3 (x_3 - x_5)^3 (x_4 - x_5)
\end{equation*}
and
\begin{equation*}
\phi^{(3,2)}(p(D)) = x_1^5 x_2^3 x_3^3 x_4^1 + 
\text{terms whose exponent partitions are $\prec (5,3,3,1)$}.
\end{equation*}

Lemma~\ref{extendedtriangularity} implies that $\phi^{(\ell, m)}$ is
surjective, and dimension counting implies that $\phi^{(\ell, m)}$ is
a graded $S_n$-module homomorphism.  The $S_n$-isomorphism in
Theorem~\ref{extension} follows.

To prove the remainder of Theorem~\ref{extension}, we need to show
that $V_n^{(\ell,m)}(k) = \Sym^k(V_n^{(\ell,m)}(1))$ for $k < n$. This
follows at once from Lemma~\ref{lem:tutteCoeff}, since the Hilbert
series of $V_n^{(\ell,m)}$ is the Tutte polynomial
evaluation $$q^{\ell\binom{n}{2} +
  mn-n}T_{K_{n+1}^{(\ell,m)}}(1,1/q).$$ For this see \cite{PostnikovShapiro}. For the statement about the
top degree, we take the elements in $V^{(\ell,\ell)}_n(\textup{top})$,
divide them all by $ \prod_{1 \leq i < j \leq n+1} (x_i-x_j)^{\ell
  -1}, $ which yields an equivariant isomorphism with
$V_n(\textup{top}) \otimes \sign^{\otimes (\ell -1)}$. By
Theorem~\ref{extremedegrees} this is $\Lie_n \otimes \sign^{\otimes
  \ell}$. The remainder of Theorem~\ref{extension} is now proved.
\end{proof}

\section{Concluding Remarks}
\label{Concluding Remarks}

In this paper we constructed a graded $S_{n+1}$-module $V_n$ which satisfies 
$\Res^{S_{n+1}}_{S_n}(V_n) \cong_{S_n} \Park_n$.  
While we know the $S_{n+1}$-structure of $V_n$ in extreme degrees, the full
$S_{n+1}$-structure remains unknown.

\begin{problem}
\label{extendedfrobeniuscharacter}
Give a nice expression for the graded $S_{n+1}$-Frobenius character of
$V_n$.
\end{problem}

Problem~\ref{extendedfrobeniuscharacter} may have an answer in terms
of free Lie algebras.  Let $\mathsf{Lie}_{n+1}$ be the free Lie
algebra on the generators $x_1, \dots, x_{n+1}$.  The group $S_{n+1}$
acts on $\mathsf{Lie}_{n+1}$ by subscript permutation.  By keeping
track of the multiplicities of the $x_i$, the $S_{n+1}$-module
$\mathsf{Lie}_{n+1}$ carries the structure of an
$\mathbb{N}^{n+1}$-graded vector space.  The $(1, \dots, 1)$-component
of this vector space is stable under the action of $S_{n+1}$ and is
known to carry the Lie representation, or $V_n(\mathrm{top}) \otimes
\mathrm{sign}$.  Lower degrees of $V_n$ may also embed naturally
inside free lie algebras.

In this paper we showed that the action of $S_n$ on $\Park_n$ extends to a graded
module $V_n(k)$ of $S_{n+1}$.  
We identified the top degree $V_n(\mathrm{top})$ of this extended action
with the twisted Lie representation $\Lie_n \otimes \sign$ of
$S_{n+1}$.  Whitehouse \cite{Whitehouse} proved that the
representation $\Lie_n$ extends to $S_{n+2}$.  
This suggests the
following problem.

\begin{problem}
\label{extend-two}
What is the maximum value of $r$ for which the action of $S_n$ on $\Park_n$ extends
to an action of $S_{n+r}$?  
For fixed $n$ and $k$, what is the maximum value of $r$ for which the action of $S_{n+1}$
on $V_n(k)$ extends to $S_{n+r}$?
\end{problem}

We have some computer evidence 
(see \cite{B} for the relevant Mathematica code)
that the value of $r$ in the first question may be greater than $1$ for
any $n$.  
For $n = 5, 6, 7$, the action of $S_n$ on $\Park_n$ extends to an action of $S_{10}$.  However,
the action of $S_5$ on $\Park_5$ does not extend to an action of $S_{11}$.

Any answer to the second question will depend on both $n$ and $k$.  By 
Whitehouse's result, for any $n > 0$, the $k$-value $k = {n \choose 2}$ gives rise to an
extension degree $r$ of at least two.  Also, since $V_n(0)$ is the trivial representation of
$S_{n+1}$ for any $n$, if $k = 0$ one can take $r = \infty$.  On the other hand, if $k = 1$ we have
that $V_n(1)$ is the reflection representation of $S_{n+1}$.  For $n > 3$, this representation is not 
the restriction of any $S_{n+2}$ module.

The results of this paper and that of Whitehouse \cite{Whitehouse} motivate the following problem
which in the opinion of the authors has received surprisingly little attention.

\begin{problem}
\label{does-it-extend}
Let $M$ be an $S_n$-module.  Give a nice criterion for deciding whether $M$ extends to a 
representation of $S_{n+1}$.
\end{problem}

Very few irreducible representations of $S_n$ extend to $S_{n+1}$.  Indeed,
if $S^{\lambda}$ is the irreducible representation of $S_n$ labeled by a partition
$\lambda \vdash n$, then $S^{\lambda}$ extends to $S_{n+1}$ if and only 
if $\lambda$ is a `near rectangle', i.e. a rectangular partition with $n+1$ boxes minus its outer corner.

On the other hand, an `asymptotically nonzero fraction' of
$S_n$-modules extend to $S_{n+1}$.  More precisely, recall that the
$\mathbb{Z}$-module $R_n$ of class functions on $S_n \rightarrow
\CC$ has basis given by the set of irreducible characters $\{
S^{\lambda} \,:\, \lambda \vdash n \}$ (where we identify modules with
characters).  The $\mathbb{Z}$-linear map $\psi: R_{n+1} \rightarrow R_n$ induced
by restriction is surjective.  Indeed, if $\lambda = (\lambda_1,
\dots, \lambda_k) \vdash n$, form a new partition $\lambda^+ :=
(\lambda_1 + 1, \lambda_2, \dots, \lambda_k) \vdash n+1$ by increasing
the first part of $\lambda$ by one.  By the branching rule for 
symmetric groups, the restriction
$\Res^{S_{n+1}}_{S_n}(S^{\lambda^+})$ has the form
\begin{equation*}
\Res^{S_{n+1}}_{S_n}(S^{\lambda^+}) \cong_{S_n} S^{\lambda} \oplus \cdots ,
\end{equation*}
where the elipsis denotes a direct sum of irreducibles corresponding
to partitions $> \lambda$ in the dominance order. The surjectivity of $\psi$ follows.  

On the level of representations, the fact that $\psi$ is surjective
means that the set $C_n$ of $S_n$-modules which extend to $S_{n+1}$
forms a full rank cone within the integer cone of representations of
$S_n$.  One way to interpret Problem~\ref{does-it-extend} would be to
describe the extremal rays and/or facets of $C_n$.  Identifying
representations of $S_n$ with points in $\mathbb{N}^{p(n)}$, where
$p(n) = \#\{ \lambda \,:\, \lambda \vdash n \}$ is the partition
number, we could also ask for the size of $C_n$ by asking for the
limit $\lim_{m \rightarrow \infty} \frac{ \#(C_n \cap \{0, 1, \dots,
  m\}^{p(n)})}{(m+1)^{p(n)}}$.  The fact that $\psi$ is surjective
means that this limit is nonzero, but we have no conjecture as to its
value.

Since $\psi$ is surjective, every representation of $S_n$ is a
restriction of a virtual $S_{n+1}$-module. In a slightly different
direction, one could ask that a sufficiently large multiple of a
representation extend. For arbitrary $m \leq n$, there is some integer
$M$ such that $\CC[S_m]^{\oplus M}$ extends to an $S_{n+1}$-module, by
a result of Donkin \cite{donkin}. Indeed, take the canonical embedding
\[S_m \subset S_{n+1} \subset\operatorname{GL}_{n+1}(\CC)\] in the
canonical way. Donkin asserts that there is a finite dimensional
rational $\operatorname{GL}_{n+1}(\CC)$-module $V$ such that
$\Res^{\operatorname{GL}_{n+1}(\CC)}_{S_n} V \approx \CC[S_n]^{\oplus
  m}$. It follows that $\Res^{\operatorname{GL}_{n+1}(\CC)}_{S_{n+1}}
V$ is the desired extension.

Problem~\ref{does-it-extend} is unsolved even for the coset
representations $M^{\lambda}$.  Not all of these representations
extend (if they did, then any direct sum of coset representations such
as the parking representation would extend automatically).  For
example, the representation $M^{(3,2,2)}$ of $S_7$ does not extend to
a representation of $S_8$ (as can be checked by computer using \cite{B}).  
However, the representation $M^{\lambda}$
of $S_n$ extends to a representation of $S_{n+1}$ for all $\lambda
\vdash n$ and $0 \leq n \leq 6$.

Many  variations on Problem~\ref{does-it-extend} are
possible. One could ask for a nice way of
determining the greatest integer $k$ such that an $S_n$-module $M$
extends to $S_{n+k}$.  

Also, one could ask whether a given {\it permutation} representation of $S_n$ extends to $S_{n+1}$.  
 For $n \leq 5$, the permutation action of $S_n$ on $\Park_n$ extends to 
 a permutation action of $S_{n+1}$; we are not 
 sure whether this permutation action extends in general.  A seeming difficulty with this question is that one
 would {\it a priori} need to consider restrictions to $S_n$ of the action of $S_{n+1}$ on the set 
 of cosets $S_{n+1} / G$ for {\it any} subgroup $G$ of $S_{n+1}$.  
 
 A more combinatorial `permutation' version of 
 Problem~\ref{does-it-extend} 
 can be obtained by asking which permutation representations of $S_n$ are restrictions
 of permutation representations of $S_{n+1}$ which are direct sums of the modules
 $\{M^{\lambda} \,:\, \lambda \vdash n+1 \}$.  The parking representation $\Park_4$ does not satisfy this property.
 To see this, one uses the fact
 that for $\lambda \vdash n+1$, 
 $\Res^{S_{n+1}}_{S_n}(M^{\lambda}) = \bigoplus_{\mu} M^{\mu}$, where $\mu$ ranges over all partitions
 obtained by subtracting $1$ from any nonzero part of $\lambda$ and sorting the resulting sequence into
 weakly decreasing order.  It is a direct computation that 
 $\Park_4 = M^{(1,1,1,1)} \oplus 6 M^{(2,1,1)} \oplus 2 M^{(2,2)} \oplus 4 M^{(3,1)} \oplus M^{(4)}$
 is not a $\mathbb{N}$-linear combination of the seven modules
 $\{ \Res^{S_5}_{S_4}(M^{\lambda}) \,:\, \lambda \vdash 5 \}$.
 
 Problem~\ref{does-it-extend} could also be interesting
in positive characteristic or for towers of linear or Weyl groups
other than symmetric groups.

\section{Acknowledgements}
\label{sec:ack}

B. Rhoades was partially supported by the NSF grant DMS-1068861.
A. Berget was partially supported as VIGRE Fellow at UC Davis by NSF
grant DMS-0636297.

The authors thank Igor Pak for early discussions on determining the
$S_{n+1}$-module structure of $V_n$ and John Shareshian for several
discussions on extendability of representations.  The authors also thank
an anonymous referee for pointing out a (corrected) flaw
in Lemma~\ref{triangularity}.

\end{document}